\documentclass[12pt]{amsart}
\usepackage[left=3cm,right=3cm,top=4cm,bottom=4cm]{geometry}
\usepackage{amsmath}
\usepackage{amsfonts}
\usepackage{amssymb}
\usepackage{mathtools}
\usepackage{blkarray}
\usepackage{lmodern}
\usepackage{blkarray, bigstrut}
\usepackage{booktabs}
\usepackage{mathrsfs}
\usepackage{xfrac} 
\usepackage[colorlinks=true]{hyperref}
\usepackage{bm}
\usepackage{lipsum}
\usepackage{scalerel}
\usepackage{stackengine,wasysym}

\newcommand\reallywidetilde[1]{\ThisStyle{%
  \setbox0=\hbox{$\SavedStyle#1$}%
  \stackengine{-.1\LMpt}{$\SavedStyle#1$}{%
    \stretchto{\scaleto{\SavedStyle\mkern.2mu\AC}{.5150\wd0}}{.6\ht0}%
  }{O}{c}{F}{T}{S}%
}}

\def\test#1{$%
  \reallywidetilde{#1}\,
$\par}

\newcommand{\proj}{{\textnormal{proj}}}

\newtheorem{theorem}{Theorem}[section]
\newtheorem{lemma}[theorem]{Lemma}
\newtheorem{corollary}[theorem]{Corollary}

\theoremstyle{definition}
\newtheorem{definition}[theorem]{Definition}
\newtheorem{example}[theorem]{Example}
\theoremstyle{remark}
\newtheorem{remark}[theorem]{Remark}
\theoremstyle{remark}
{
\newtheorem*{notation}{Notation}
}
\numberwithin{equation}{section}

\newcommand{\CC}{{\mathbb C}}

\newcommand{\NN}{{\mathbb N}}

\newcommand{\bigslant}[2]{{\raisebox{.2em}{$#1$}\left/\raisebox{-.2em}{$#2$}\right.}}

\hypersetup{
  pdftitle={},
  pdfauthor={},
  pdfsubject={},
  pdfkeywords={},
  colorlinks=true,
  breaklinks=true,
  bookmarksopen=true,
  bookmarksnumbered=true,
  pdfpagemode=UseOutlines,
  plainpages=false,
  linkcolor=black,
  citecolor=black,
  anchorcolor=black,
  urlcolor  = black
  }
\begin{document}

\title{On Two Approaches to Cluster Structures on Partial Flag Varieties}

\author{Fayadh Kadhem}
\address{Faculty of Professional Studies\\
Bahrain Polytechnic\\
Isa Town, Bahrain}
\email{fayadh.kadhem@polytechnic.bh}

\date{\today}
\maketitle

\begin{abstract}
Continuing our previous work, this paper closely studies the relationship between the cluster algebra structures on the coordinate ring of Schubert cells and those on the coordinate ring of partial flag varieties. We give a finite-type classification for these cluster structures and point out several results that were left open in our previous work.\\
\end{abstract}

\maketitle

\section{Introduction}
\label{introduction}

Cluster algebras were introduced in the early 2000s by Fomin and Zelevinsky \cite{FZ0} in their research on total positivity. Roughly speaking, cluster algebras are commutative algebras generated inside an ambient field by a set of variables that are divided into two kinds: \textit{mutable} and \textit{frozen}. They are also endowed with a \textit{skew-symmetrizable} matrix that allows one to define recursive relations among the mutable variables and proceed from one cluster to another.

Amazingly, mathematicians quickly discovered many strong connections between cluster algebras and several other areas of mathematics. One important connection with Lie theory arises through their finite type classification by the classical Dynkin diagrams \cite{FWZ2,FZ2}. In the context of algebraic geometry, Scott showed in 2006 that the homogeneous coordinate ring of the Grassmannian can be endowed with a cluster algebra structure \cite{S}, allowing a much better understanding of its properties. This opened the door to many subsequent developments. A significant breakthrough came in 2008 by Gei{\ss}, Leclerc, and Schr\"{o}er (GLS). They showed in \cite{GLS1,GLS2} that the coordinate ring of a Schubert cell admits a cluster algebra structure for types $A_n$ and $D_4$. Moreover, they showed that this structure can be lifted to give a cluster structure (after suitable localization) on partial flag varieties. This program was later extended to more general types, covering all simply-laced cases \cite{GLS3}. Consequently, Demonet \cite{D} was able to use their construction to lift the cluster algebra structure to the coordinate rings of partial flag varieties.

In a different direction, Goodearl and Yakimov (GY) constructed a large class of cluster algebras in connection with Poisson algebras \cite{GY}. As a consequence of their work, it was proved that the coordinate ring of any Schubert cell, regardless of type, admits a cluster algebra structure. This work served as the foundation for constructing another cluster algebra structure on the coordinate ring of a partial flag variety. The explicit cluster structure was developed mainly in \cite{Kadhem, Kadhem2}.\

One important question that arises is how the cluster algebra structures of GLS and GY are related. Consequently, what can one conclude about the relationship between the cluster structures of \cite{D} and \cite{Kadhem}?\

In this paper, we prove that the finite type classifications of GLS and GY (and consequently those of \cite{D} and \cite{Kadhem}) are exactly the same, and we point out several consequences of this fact.

The paper is organized as follows. Section 2 provides an overview of cluster algebras. In Section 3, we give a brief review of flag varieties. Section 4 introduces the explicit cluster structures that are the main focus of the paper. Section 5 contains the finite type classification, and Section 6 discusses some consequences of the results obtained.

\section{Cluster Algebras}
In this section, we give a quick overview of cluster algebras and their main results. The reader who is interested in more details is recommended to see \cite{FWZ}, \cite{FZ}, or \cite{W}.

\begin{definition}
Let $n$ and $m$ be positive integers such that $n \leq m$. A \textit{seed} is a pair $(\textnormal{\textbf{x}},B)$ such that $\textnormal{\textbf{x}}$ is a tuple of algebraically independent variables $(x_1,...,x_n,x_{n+1},...,x_m)$ and $B$ is a skew-symmetrizable $m \times n$ matrix, that is, a matrix whose north $n\times n$ submatrix is skew symmetric up to multiplying each row $r_i$ by a nonzero integer $c_i$. The first $n$ variables are called \textit{mutable}, the remaining $m-n$ ones are called \textit{frozen}, and the matrix $B$ is called the \textit{exchange matrix}.
\end{definition}

\begin{notation}
In the previous definition, if $x_i$ is a variable of the tuple $\textnormal{\textbf{x}}$, we may express this by writing $x_i \in \textnormal{\textbf{x}}$.
\end{notation}

The previous definition suggests a hidden distinction between the first $n$ variables and the rest $m-n$ ones. The following definition clarifies this distinction.

\begin{definition} \label{mutation}
Let $(\textnormal{\textbf{x}},B)$ be a seed and denote the entries of $B$ by $b_{ij}$. Let $k$ be an index of a mutable variable. The \textit{mutation} $\mu_k$ at $k$ is a transformation to another seed $(\textnormal{\textbf{x}}',B')$ such that
$$x_k'= \dfrac{\prod_{b_{ik}>0} x_i^{b_{ik}} + \prod_{b_{ik}<0} x_i^{-b_{ik}}}{x_k}$$
and $x_i=x_i'$ for all $i \neq k$. Also, the exchange matrix $B'$ of the new seed is given by

\begin{equation}
    \label{eq1}
b'_{ij}=\begin {cases}
-b_{ij}, & \text{if}\ i=k \text{ or } j=k,\\ 
b_{ij}+\dfrac{|b_{ik}|b_{kj} + b_{ik}|b_{kj}|}{2}, & \text{otherwise}.\\
\end{cases}
\end{equation}
In terms of notation, we may write $(\textnormal{\textbf{x}}',B')$ as $\mu_k(\textnormal{\textbf{x}},B)$.
\end{definition}

\begin{remark}
In many situations, it is convenient to replace the skew-symmetrizable matrix appearing in the previous definition with a \textit{quiver} $Q$, namely a directed graph consisting of $n$ \textit{mutable} vertices and $(m-n)$ \textit{frozen} vertices. We require that $Q$ has no loops, no oriented $2$-cycles, and no arrows between frozen vertices. Associated to such a quiver is an $m \times n$ skew-symmetrizable matrix $\widetilde{B}(Q)$ whose entries are given by
$$b_{ij}=\begin {cases} \# (i\rightarrow j), & \text{if}\ i > j,\\ 0, & \text{if}\ i=j,\\ - \# (i \leftarrow j), & \text{if}\ i < j; \end{cases}$$
where $\#(i\rightarrow j)$ denotes the number of arrows from $i$ to $j$, while $\#(i\leftarrow j)$ denotes the number of arrows from $j$ to $i$.

One can mutate directly at a mutable vertex $k$ in a quiver $Q$ to obtain a new quiver $\mu_k(Q)$ as follows:

\begin{enumerate}

\item For every oriented path $i \rightarrow k \rightarrow j$, add an arrow $i \rightarrow j$.

\item Reverse all arrows connected to the vertex $k$.

\item Remove all oriented $2$-cycles that arise after the previous steps, repeating this process until no such cycles remain.

\end{enumerate}

The resulting quiver $\mu_k(Q)$ again has no loops and no oriented $2$-cycles. It is not hard to verify that the mutation on a quiver and the mutation on the matrix associated to a quiver are equivalent.
\end{remark}

\begin{remark}
One can easily verify that $\mu_k$ is an involution, that is, $$\mu_k(\mu_k(\textnormal{\textbf{x}},B))=(\textnormal{\textbf{x}},B).$$
\end{remark}

\begin{definition}
Let $(\textnormal{\textbf{x}},B)$ be a seed. The cluster algebra $\mathcal{A}$ attached to the given seed is the subring of $\CC(x_1,...,x_n,...,x_m)$ generated by the frozen variables and all the mutable variables, that is, the mutable variables of the original seed or any seed induced by any sequence of mutations at any indices. In this context, we may denote $\mathcal{A}$ by $\mathcal{A}(\textnormal{\textbf{x}},B)$ and call $(\textnormal{\textbf{x}},B)$ the \textit{initial seed}.
\end{definition}

\begin{remark}
It is straightforward to see that the cluster algebra induced by some seed is exactly the same cluster algebra induced by any seed mutation or any sequence of mutations of the initial one.
\end{remark}

\begin{definition}
The \textit{rank} of a cluster algebra $\mathcal{A}$ is the number of mutable variables of its initial seed or any mutation of it. If the number of all the obtained seeds is finite, we say that $\mathcal{A}$ is \textit{of finite type}. Otherwise, it is called \textit{of infinite type}.
\end{definition}

\begin{theorem}\label{finite type}
A cluster algebra $\mathcal{A}$ is of finite type if and only if the Cartan counterpart of one of its seeds is a Cartan matrix of finite type, that is, of type $A_n$, $B_n$, $C_n$, $D_n$, $E_6$, $E_7$, $E_8$, $F_4$, or $G_2$.
\end{theorem}
\begin{proof}
See \cite{FWZ2} or \cite{FZ2}.
\end{proof}

\section{Flag varieties and GLS tilde map}
This section introduces the required results about flag varieties and related topics. For more details, the reader is referred to \cite{GLS1, GLS2, GLS3} or \cite{GSV}.\\

Let $G$ be a simply-connected semisimple complex algebraic group and let $I$ be its Dynkin diagram vertex set. Let $J$ and $K$ be a pair of disjoint sets such that $J\sqcup K=I$.

\begin{definition}
A subgroup $P$ of $G$ is called \textit{parabolic} if it is closed and contains a Borel subgroup.
\end{definition}

\begin{example}
In this example, we give a couple of important parabolic subgroups that will be used throughout the whole paper.
\begin{enumerate}
    \item Any Borel subgroup $B$ is parabolic.
    \item Let $B,B^-$ be a pair of Borel subgroups and denote their unipotent radicals by $N,N^-$, respectively. For each of $N,N^-$, there are one-parameter root subgroups indexed by $I$ that form a set of distinguished generators. Denote by $x_i(t)$ $(i\in I, t \in \CC)$ and $y_i(t)$ the ones of $N$ and $N^-$, respectively. The subgroup $P_K$ generated by $B$ and $y_k(t)$ is parabolic and called the \textit{standard parabolic subgroup} associated to $B$. Analogously, the subgroup $P_K^-$ generated by $B^-$ abd $x_k(t)$ is parabolic.
\end{enumerate}
\end{example}

\begin{definition}
A \textit{(partial) flag variety} is a quotient $G/P$, where $P$ is a parabolic subgroup.
\end{definition}

\begin{remark}
Each parabolic subgroup is conjugate to a unique standard parabolic subgroup. Consequently, the study of flag varieties is reduced to the ones of the form $G/P_K^-$, which will be our main focus on this paper.
\end{remark}

\begin{remark}
For a dominant weight $\lambda$, denote the corresponding finite-dimensional irreducible $G$-module of $G$ with highest weight $\lambda$ by $L(\lambda)$. Let $L(\lambda)^*$ denote the right $G$-module obtained by twisting the action of $G$. Naturally, the flag variety $G/P_K^-$ can be embedded as a closed subset of the product of projective spaces $$\prod_{j \in J}\mathbb P \big(L(\varpi_j)^*\big),$$
where $\varpi_j$ is a fundamental weight of $G$.
\end{remark}

\begin{remark}
Let $\Pi_J$ be the monoid of integral dominant weights that is isomorphic to $\NN^J$ and defined by
$$\Pi_J=\bigg\{ \lambda=\sum_{j \in J} a_j \varpi_j \mid a_j \in \NN \bigg\}.$$
The coordinate ring $\CC[G/P_K^-]$ is $\Pi_J$-graded. Indeed,
$$\CC[G/P_K^-]=\bigoplus_{\lambda \in \Pi_J} L(\lambda).$$
Consequently, $\CC[G/P_K^-]$ can be identified with the subalgebra $\CC[G/N^-]$ generated by the homogeneous elements of degree $\varpi_j$, where $j \in J.$
\end{remark}

\begin{remark}[Generalized minors]
If $G$ is of type $A$, then a \textit{(flag) minor} is a regular irreducible function on $\mathbb{C}[G]$ defined as follows: for any subset $X \subset [1,n] := \{1, \dots, n\}$ and any matrix $x \in G$, the minor $\Delta_X(x)$ is defined to be the determinant of the submatrix of $x$ whose rows are indexed by $X$ and whose columns are the first $|X|$ columns. 

Fomin and Zelevinsky generalized this notion in \cite{FZ0} to the concept of a \textit{(generalized) minor} $\Delta_{u\varpi_j, w\varpi_j}$, where $u, w$ belong to the Weyl group $W$. In type $A$, the notions of flag minors and generalized minors coincide. However, the definition of generalized minors extends to arbitrary types.
\end{remark}

\begin{remark}
For each simple reflection $s_i \in W$, define $\overline{s_i} := \exp(f_i)\exp(e_i)\exp(f_i)$. If $w = s_{i_1} \cdots s_{i_r}$ is a reduced expression of $w \in W$, then set $\overline{w} := \overline{s_{i_1}} \cdots \overline{s_{i_r}}$. 

Let $G_0 = N^-HN$ denote the open subset of $G$ consisting of elements admitting a Gaussian decomposition. That is, each $x \in G_0$ can be uniquely written as
\[
x = [x]_{-} [x]_0 [x]_+,
\]
where $[x]_{-} \in N^-$, $[x]_0 \in H$, and $[x]_+ \in N$.\\
Let $V_i^+$ be the irreducible representation of $G$ with highest weight $\varpi_i$ and highest weight vector $v_i^+$. For any $h \in H$, the vector $v_i^+$ is an eigenvector: 
\[
h v_i^+ = [h]^{\varpi_i} v_i^+ \quad \text{for some } [h]^{\varpi_i} \in \mathbb{C} \setminus \{0\}.
\] 
This leads to the following definition introduced by Fomin and Zelevinsky in \cite{FZ0}.
\end{remark}

\begin{definition}
For $u, v \in W$ and $i \in I$, the \textit{(generalized) minor} is the regular function on $G$ defined by
\[
\Delta_{u\varpi_i, v\varpi_i}(x) := [\overline{u}^{-1} x \, \overline{v}]_0^{\varpi_i}.
\]
\end{definition}

\begin{remark}
For each $w \in W$, the multi-degree of the minors of the form $\Delta_{\varpi_j,w(\varpi_j)}$ in $\CC[G/P_K^-]$ is $\varpi_j$.
\end{remark}

\begin{remark}
The minors can be used to define the affine coordinate ring of the unipotent radical cell $N_K$ of $P_K$ in the following two equivalent ways
\begin{enumerate}
    \item the homogeneous elements of degree 0 of the localization of the homogeneous coordinate ring of the corresponding flag variety by the elements $\Delta_{\varpi_j,\varpi_j}$ where $j \in J$; or
    \item the quotient of $\CC[G/P_K^-]$ by the ideal generated by the elements $\Delta_{\varpi_j, \varpi_j} - 1$ for all $j \in J$.\\
\end{enumerate}
In symbols, we have that
\begin{equation}\label{localization}
  \CC[N_K]=\left\{\dfrac{f}{\prod_{j \in J} \Delta_{\varpi_j,\varpi_j}^{a_j}} \mid f\in L \bigg(\sum_{j\in J} a_j \varpi_j \bigg) \right\},
  \end{equation}
or equivalently,
\begin{equation}\label{quotient}
  \mathbb C [N_{K}]=\bigslant{{ \mathbb C [G/P_{K}^{-}]}}{(\Delta_{\varpi_j,\varpi_j}-1)}_{j \in J}.
  \end{equation}
The second identification gives a natural canonical projection
$$\proj_J: \mathbb C [G/P_K^-] \to \mathbb C [N_K].$$
The restriction of the projection map to each homogeneous piece $L(\lambda),$ for $\lambda \in \Pi_J$, gives an injection $L(\lambda) \xhookrightarrow{} \mathbb C[N_K]$.
\end{remark}

\begin{remark}
There is a standard partial order $\preceq$ on the set $\Pi_J$ defined as follows:
\[
\lambda \preceq \mu \quad \Longleftrightarrow \quad \mu - \lambda \in \mathbb{N} \{\varpi_j \mid j \in J\},
\]
that is, the difference $\mu - \lambda$ is an $\mathbb{N}$-linear combination of the fundamental weights $\varpi_j$ for $j \in J$.
\end{remark}

\begin{lemma}[GLS tilde map] \label{GLS tilde map}
There is a map $\text{\test{\textnormal{ }\cdot \textnormal{ }}}:\CC[N_K] \to \CC[G/P_K^-]$ that lifts each $f \in \CC[N_K]$ uniquely to a homogeneous element $\tilde{f} \in \CC[G/P_K^-]$ in which $\proj(\tilde{f})=f$ and the multi-degree of it is minimal up to $\preceq$.
\end{lemma}

\begin{lemma}\label{summation lift}
The map tilde $\text{\test{\textnormal{ }\cdot \textnormal{ }}}$ commutes with the usual multiplication of $\CC[N_K]$, that is, for $f,g \in \CC[N_K]$, we have $\widetilde{f \cdot g}=\widetilde{f} \cdot \widetilde{g}$. Moreover, if $a_j(f+g)= \max \{a_j(f), a_j(g) \}$ for all $j \in J$, then
$$\widetilde{f+g}=\mu \widetilde{f} + \nu \widetilde{g},$$
where $\mu$ and $\nu$ are relatively prime monomials in the variables $\Delta_{\varpi_j, \varpi_j}$, $(j \in J)$.
\end{lemma}

\noindent Consequently, as indicated in \cite{Kadhem2}, this allows us to rewrite equation (\ref{localization}) as
\begin{equation}\label{localization 2}
  \CC[N_K]=\left\{ \dfrac{f}{\prod_{j \in J} \Delta_{\varpi_j,\varpi_j}^{a_j}} \mid f\in L \bigg(\sum_{j\in J} a_j \varpi_j \bigg) \textnormal{ and $a_j$ is minimal} \right\}.
  \end{equation}

\section{Cluster structures}
This section takes a look at constructions of cluster algebra structures on $\CC[N]$ and $\CC[G/P]$, where $N$ is a unipotent Suchbert cell and $G/P$ is its corresponding partial flag variety. For more details, the reader is referred to \cite{GLS1, GLS2, GLS3}, \cite{GY}, and \cite{Kadhem, Kadhem2}.

\begin{remark}
It has been proved in \cite{GLS1, GLS2, GLS3, GLS4} by Gei{\ss}, Leclerc and Schr{\"{o}}er that $\CC[N_K]$ admits a cluster algebra structure for the simply-laced groups. However, Goodearl and Yakimov in \cite{GY} were able to give a general construction that works for any arbitrary type of group $G$.
\end{remark}

\begin{definition}
For a word $w=s_{i_1}...s_{i_n} \in W$ define the following
\begin{align*}
p(k) &:=\begin{cases}
\textnormal{max}\{j < k \textnormal{ }| \textnormal{ } i_j=i_k \},& \text{if such } j \text{ exists;}\\
- \infty, & \text{otherwise.}
\end{cases}\\
s(k)&:=\begin{cases}
\textnormal{min}\{j > k \textnormal{ }| \textnormal{ } i_j=i_k \},& \text{if such } j \text{ exists;}\\
\infty, & \text{otherwise.}
\end{cases}\\
S(w)&:=\{i \in I \mid s_i \leq w \}=\{i \in I \mid i=i_k \textnormal{ for some } k \in [1,m] \}.
\end{align*}
\end{definition}

These definitions are critical for the following theorem:

\begin{theorem}\label{seed construction}
There is a canonical cluster algebra structure on $\CC[N_K]$ such that its initial seed is given by the following data:
\begin{itemize}
    \item The initial seed variables are $D_{\varpi_{i_k},w_{\leq k} \varpi_{i_k}}=\proj_J\big( {\Delta_{\varpi_{i_k},w_{\leq k} \varpi_{i_k}}}\big)$.
    \item The variable is frozen if and only if it is indexed by a $k$ such that $s(k)=\infty$.
    \item The initial exchange matrix $\widetilde{B}^w$ is of size $m \times (m-|S(w)|)$, where $m$ is the size of the subword $w_K$ that generates $N_K$. The $j \times k$ entry is given by
$$(\widetilde{B}^w)_{jk} = \begin{cases}
1, & \text{if } j=p(k), \\
-1, & \text{if } j=s(k),\\
a_{i_j i_k}, & \text{if } j<k<s(j)<s(k),\\
-a_{i_j i_k}, & \text{if } k<j<s(k)<s(j),\\
0, & \text{otherwise,}
\end{cases}$$
where $a_{i_j i_k}$ is the $i_j \times i_k$ entry of the Cartan matrix of the same type of $G$.
\end{itemize}
\end{theorem}
\begin{proof}
This is a consequence of the results of the work of GY in \cite{GY2, GY, GY3}. A formal proof can be found in Theorem 4.8 in \cite{Kadhem}.
\end{proof}

\begin{corollary}
The quotient algebra $\bigslant{{ \mathbb C [G/P_{K}^{-}]}}{(\Delta_{\varpi_j,\varpi_j}-1)}_{j \in J}$ admits a canonical cluster algebra structure.
\end{corollary}

Now, this cluster algebra can be lifted to another important cluster algebra using the following results.

\begin{remark} \label{cluster mutation}
The mutation at $k$ for any seed $({{\textnormal{\textbf{x}}}},B)$ of the cluster algebra $\mathcal{A}= \CC[N_K]$ gives the following exchange relation
$$x_k x_k' = M_k + L_k,$$
where $M_k,L_k$ are monomials in the variables $x_1,...,x_{k-1},x_{k+1},...,x_n.$ By Lemma \ref{summation lift}, this lifts to the equation
$$\widetilde{x_k x'_k}=\mu(k)\widetilde{M_k}+\nu(k) \widetilde{L_k},$$
where $\mu(k)$ and $\nu(k)$ are relatively prime monomials in the vatiables $\Delta_{{\varpi_{j}},{\varpi_{j}}}$, in which $j \in J$. This means that $\mu(k)$ and $\nu(k)$ can be expressed as
$$\mu(k) = \prod_{j \in J} \Delta_{{\varpi_{j}},{\varpi_{j}}}^{\alpha_j} \quad \textnormal{ and } \quad \nu(k) = \prod_{j \in J} \Delta_{ {\varpi_{j}}, {\varpi_{j}}}^{\beta_j},$$
where $\min \{\alpha_j, \beta_j \}=0$ for all $j$.\\

Note here that the monomial $M_k$ is the one corresponding to the positive $b_{ik}$'s in Definition \ref{mutation}, while the monomial $L_k$ is the one corresponding to the negative $b_{ik}$'s.
\end{remark}

\begin{definition}
Let $(\textnormal{\textbf{x}},B)$ be a seed of the cluster algebra $\mathcal{A}=\CC[N_K]$. The \textit{homogenization} of $(\textnormal{\textbf{x}},B)$ is the seed $(\widehat{\textnormal{\textbf{x}}},\widehat{B})$ consisting of:
\begin{itemize}
    \item The variables $\widetilde{x}$ obtained by applying the GLS tilde map on each $x \in \textnormal{\textbf{x}}$.
    \item The minors ${\Delta_{\varpi_j,\varpi_j}}$, where ${j \in J}$.
    \item The exchange matrix $\widehat{B}$ that extends the matrix $B$ by $|J|$ row labeled by the elements $j \in J$ such that
\[
  \widehat{b}_{jk} =
  \begin{cases}
\beta_j , & \text{if $\beta_j \neq 0$;} \\
- \alpha_j,& \text{else,}
  \end{cases}
  \]
where $\alpha_j$ and $\beta_j$ are as specified in Remark~\ref{cluster mutation}.
\end{itemize}
The mutable variables of the homogenization are the $\widetilde{x}$ coming from mutable $x \in \textnormal{\textbf{x}}$. All the other variables are frozen.
\end{definition}

\begin{theorem}
If $(\textnormal{\textbf{x}},{B})$ and $(\textnormal{\textbf{y}},{C})$ are two seeds of the cluster algebra $\mathcal{A}=\CC[N_K]$ related by a sequence of mutations, then $ (\widehat{\textnormal{\textbf{x}}},\widehat{B})$ and $(\widehat{\textnormal{\textbf{y}}},\widehat{C})$ are related by exactly the same sequence of mutations.
\end{theorem}
\begin{proof}
It is enough to prove the theorem for one single mutation of the seed $(\textnormal{\textbf{x}},{B})$. This was done in Theorem 5.8 by \cite{Kadhem}.
\end{proof}

\begin{corollary}
The coordinate ring $\CC[G/P_K^-]$ contains a cluster algebra $\widehat{\mathcal{A}}$
whose initial seed is given by the pair
$$\bigg( \big\{ \widetilde{ D}_{\varpi_{i_k},w_{\leq k} \varpi_{i_k}}  \big\} \sqcup \{ \Delta_{\varpi_j, \varpi_j} \mid j\in J \}, \widehat{\widetilde{B}^w} \bigg).$$
\end{corollary}

Moreover, the work of \cite{Kadhem} generalized previous results of \cite{GLS1} and showed the following:

\begin{theorem}\label{equality theorem}
The localization of the coordinate ring $\CC[G/P_K^-]$ by $\Delta_{\varpi_j, \varpi_j}$, where $j\in J $, equals the localization of the cluster algebra $\widehat{\mathcal{A}}$ by the same elements. In symbols,
$$\CC[G/P_K^-][\Delta_{\varpi_j, \varpi_j}^{-1}]_{j\in J}=\widehat{\mathcal{A}}[\Delta_{\varpi_j, \varpi_j}^{-1}]_{j\in J}.$$
\end{theorem}

\begin{remark}
It has been widely expected that $\CC[G/P_K^-]$ is indeed a cluster algebra without doing the localization of the previous theorem. See \cite{GLS1} and \cite{Kadhem}, for example.
\end{remark}

\begin{remark}
Based on the results above, it is easy to see that the cluster structure $\mathcal{A}$ of
$$\CC[N_K] \quad \textnormal{and} \quad \bigslant{{ \mathbb C [G/P_{K}^{-}]}}{(\Delta_{\varpi_j,\varpi_j}-1)}_{j \in J}$$
and the cluster algebra $\widehat{\mathcal{A}}$ are of the same type (both finite or both infinite). Hence, by construction, their finite type classification is exactly the same.
\end{remark}

\begin{remark}
Write how this cluster algebra affects the view of the elements of each of the algebras above.
\end{remark}

\section{Finite type classification}
The main goal of this section is to classify the types of the cluster algebra $\mathcal{A}=\CC[N_K]$. We will treat the types of $G$ case by case. For each type below, we will make a choice of the subset $J$ and express the longest word $w_0$ in two parts, regular and bold. The bold one will be the one that generates $N_K$ in each case. In Section 11 of \cite{GLS1}, it was proved that we only need to do the cases below, no more, no less. Despite the fact that they made a classification there, we provide a different direct approach to do the same classification. The main recipe for doing this is Theorem \ref{finite type} and Theorem \ref{seed construction}.
\subsection{Type $A_n$}
Recall that the longest word $w_0$ of type $A_n$ is given by
$$w_0=s_1 s_2 s_3...s_n s_1 s_2s_3 ... s_{n-1} ... s_1 s_2 s_1.$$
Throughout, the notation $w_{k,...,\ell}$ is used to represent the longest word of type $A_{\ell-k+1}$ indexed by $k,k+1,...,\ell$, namely,
$$w_{k,...,\ell}=s_ks_{k_+1}...s_\ell s_ks_{k_+1}...s_{\ell-1}...s_ks_{k+1}s_k.$$

\begin{enumerate}
    \item Let us choose first $J=\{2\}.$ It is not hard to see that
$$w_0=s_1 w_{3,...,n} \boldsymbol{s_2 s_3...s_ns_1s_2...s_{n-1}} $$
is a valid representation for the longest word of type $A_n$. For the bold part, $w_K=s_2 s_3...s_ns_1s_2...s_{n-1}$, it is straightforward to see that
$$s(k)=\infty \iff k \in \{n-1,n,n+1,...,2n-2\}.$$
Hence, the frozen variables are the ones indexed by $n-1,n,n+1,...,2n-2$, while the others are mutable. Consequently, the northern square submatrix of the exchange matrix $B$ attached to this cluster algebra is
 \[
\begin{blockarray}{ccccccc}
 & 1 & 2 & 3 & ... & n-3 & n-2 \\
\begin{block}{c(cccccc)}
  1&0& -1 & 0 &...& 0 & 0 \\
  2 & 1 & 0 & -1 & ... & 0 & 0 \\
  3 & 0 & 1 & 0 & ... & 0 & 0 \\
  \vdots & & & \ddots
& \ddots & \ddots \\
  n-3 & 0 & 0 & 0 & ... & 0 & -1 \\
  n-2 & 0 & 0 & 0 & ... & 1 & 0 \\
\end{block}
\end{blockarray}
\quad .\]
Clearly, the Cartan counterpart of this matrix is the Cartan matrix of type $A_{n-2}$. Hence, this cluster algebra is of finite type.\\

\item Now, we choose $J=\{ 1,2\}.$
One can verify that
$$w_0=w_{3,...,n} \boldsymbol{s_2s_3...s_ns_1s_2s_3...s_n}$$
is a valid expression for the longest word of type $A_n$. Obviously,
$$s(k)=\infty \iff k \geq n.$$
Hence, the northern $(n-1)$-square matrix of the exchange matrix $B$ is

 \[
\begin{blockarray}{ccccccc}
 & 1 & 2 & 3 & ... & n-2 & n-1 \\
\begin{block}{c(cccccc)}
  1&0& -1 & 0 &...& 0 & 0 \\
  2 & 1 & 0 & -1 & ... & 0 & 0 \\
  3 & 0 & 1 & 0 & ... & 0 & 0 \\
  \vdots & & & \ddots
& \ddots & \ddots \\
  n-2 & 0 & 0 & 0 & ... & 0 & -1 \\
  n-1 & 0 & 0 & 0 & ... & 1 & 0 \\
\end{block}
\end{blockarray}
\quad .\]
The Cartan counterpart here is of type $A_{n-1}$. Therefore, this also gives a cluster algebra of finite type.\\

\item If we choose $J=\{1,n-1\}$, then the following is a valid expression of $w_0$:
$$w_0=w_{2,...,n-2}s_n \boldsymbol{(s_1s_2) (s_{n-1}s_n) (s_3...s_n) (s_2s_3...s_{n-2}) (s_1...s_{n-3})}.$$
\noindent One can now see that
$$s(k)=\infty \iff \bigg( k \in \{n+1, n+2\} \quad \textnormal{or} \quad k\geq 2n-1 \bigg).$$
Based on this, we find that the submatrix of the exchange matrix $B$ generated by the mutable variables is
\[
\begin{blockarray}{cccccccccccccccc}
 & 1 & 2 & 3 & 4 & 5 & 6 &... & n-1 &n & n+3 &n+4 & n+5 &  ... & 2n-3 & 2n-2 \\
\begin{block}{c(ccccccccccccccc)}
  1&0 & 0 
  & 0 & 0 & 0 & 0 &...& 0 & 0 & 0 & 0 & 0 & ... & 0 & 0\\
  2 & 0 & 0 & 0 & 0 & 0 & 0  & ... & 0 & 0  & 1 & 0 & 0  & ...  &0 & 0\\
  3 & 0 & 0 & 0 & -1 & 0 & 0  & ... & 0 & 0 & 0 & 0 & 0 & ... & 0 & 0\\
  4 & 0 & 0 & 1 & 0 & -1 & 0 & ... & 0 & 0 & 0 & 0 & 0 & ... & 0 & 0\\
  5 & 0 & 0 & 0 & 1 & 0 & -1 & ... & 0 & 0 & 0 & 1 & 0 & ... & 0 & 0\\
  6 & 0 & 0 & 0 & 0 & 1 & 0 & \ddots & 0 & 0 & 0 & 0 & 1 & ... & 0 & 0\\
  \vdots & & & & & & \ddots
\\
n-1 & 0 & 0 & 0 & 0 & 0 & 0  & ... & 0 & -1 & 0 & 0 & 0 & ... & 0 & 1\\
n & 0 & 0 & 0 & 0 & 0 & 0  & ... & 1 & 0 & 0 & 0 & 0 & ... & 0 & 0\\
  n+3 & 0 & -1 & 0 & 0 & 0 & 0 & ... & 0 & 0 & 0 & -1 & 0 &... & 0 & 0 \\
  n+4 & 0 & 0 & 0 & 0 & -1 & 0 & ... & 0 &  0 & 1 & 0 & -1 & ... & 0 & 0\\
  n+5 & 0 & 0 & 0 & 0 & 0 & -1 & ... & 0 & 0 & 0 & 1 & 0 & \ddots & 0 & 0\\
  \vdots &   &  &  &  &  &  & \ddots & &  & & & \ddots\\
   & & & & & & & & & &  & & & & & -1\\
  2n-2 & 0 & 0 & 0 & 0 & 0 & 0 & ... & -1 & 0 & 0 & 0 & 0 & ... & 1 & 0 \\
\end{block}
\end{blockarray}
\quad .\]

\noindent One can now do simple row and column permutations to verify that the Cartan counterpart of this matrix is of type $A_{2n-4}$.\\

\item If $J=\{ 1,n \}$, it is not hard to verify that the longest word $w_0$ can be expressed as
$$w_0=w_{2,...,n-1} \boldsymbol{s_1s_2...s_ns_{n-1}s_{n-2}...s_2s_1}.$$
It is routine to check that
$$s(k)= \infty \iff k \in \{ n,n+1,...,2n-1 \}$$

\noindent Calculating the northern square matrix of the exchange matrix $B$, we get the zero matrix of size $n-1$. This provides the cluster algebra that is conventionally said to be of type $(A_1)^{n-1}$.\\

\item For $J=\{1,2,n\}$, it is not hard to see that
$$w_0=w_{3,...,n-1}\boldsymbol{s_1s_2...s_{n-1}s_ns_{n-1}...s_1s_2...s_{n-1}}.$$
Now,
$$s(k)= \infty \iff \Big( k=n \quad \textnormal{ or } \quad k \geq 2n-1 \Big).$$
This implies that the number of mutable variables is $2n-3$. Consequently, the northern square matrix is of size $2n-3$ and is given by
\[
\begin{blockarray}{ccccccccccc}
 & 1 & 2 & 3 & ... & n-1 & n+1 &n+2 &  ... & 2n-3 & 2n-2 \\
\begin{block}{c(cccccccccc)}
  1&0& 0 & 0 &...& 0 & 0 & 0 & ... & 0 & 0\\
  2 & 0 & 0 & 0 & ... & 0 & 0 & 0 & ... &0 & 1\\
  3 & 0 & 0 & 0 & ... & 0 & 0 & 0 & ... & 1 & 0\\
  \vdots & \vdots & & \ddots
& \ddots & \ddots & & & \reflectbox{$\ddots$} & & \vdots \\
n-2 & 0 & 0 & 0 & ... & 0 & 0 & 1 \\
  n-1 & 0 & 0 & 0 & ... & 0 & 1 & 0 & ... & 0 & 0 \\
  n+1 & 0 & 0 & 0 & ... & -1 & 0 & 0 & ... & 0 & 0\\
  \vdots & \vdots  &  &  & \reflectbox{$\ddots$} &  &  & \vdots & & & \vdots\\
  2n-3 & 0 & 0 & -1 & ... & 0 & 0 & 0 & ... & 0 & 0 \\
  2n-2 & 0 & -1 & 0 & ... & 0 & 0 & 0 & ... & 0 & 0 \\
\end{block}
\end{blockarray}
\quad .\]

\noindent Now, let $i + j = 2n - 1$. If we simultaneously swap row $r_i$ with row $r_j$ and column $c_i$ with column $c_j$, the resulting matrix is of type $A_{2n-3}$.\\

\item If $J=\{1\}$, one can use the following expression for the longest word
$$w_0=w_{2,...,n}\boldsymbol{s_1s_2...s_n}.$$
It is easily seen that this gives the zero matrix of size $n$, which we denote by $(A_1)^n$.\\


Now, we do cases specified for $n=4$, that is $A_4$:\\
\item If $J=\{2,3\}$, then one can verify that
$$w_0=s_1s_4\boldsymbol{s_2s_3s_4s_1s_2s_3s_1s_2}.$$
It is straightforward to check that:

$$s(k)=\infty \iff k \in \{3,6,7,8\}.$$

The principal submatrix generated by the mutable indices is

 \[
\begin{blockarray}{ccccc}
 & 1 & 2 & 4 & 5 \\
\begin{block}{c(cccc)}
  1&0& -1 & -1 & 1  \\
  2 & 1 & 0 & 0 & -1  \\
  4 & 1 & 0 & 0 & -1 \\
  5 & -1 & 1 & 1 & 0 \\
\end{block}
\end{blockarray}
\quad .\]

\noindent Looking at the underlying graph of the Cartan counterpart, it is clear that this is of type $D_4$.\\

\item Now, consider the case of $J=\{1,2,3\}$. We may express $w_0$ as follows:
$$w_0=s_4\boldsymbol{s_3s_2s_1s_4s_3s_2s_4s_3s_4}.$$
Obviously,
$$s(k)= \infty \iff k \in \{ 3,6,8,9 \}.$$

\noindent Consequently, the mutable principal part of the exchange matrix is

\[
\begin{blockarray}{cccccc}
 & 1 & 2 & 4 & 5 & 7\\
\begin{block}{c(ccccc)}
  1& 0 & -1 & -1 & 1 & 0 \\
  2 & 1 & 0 & 0 & -1 & 0 \\
  4 & 1 & 0 & 0 & -1 & 1 \\
  5 & -1 & 1 & 1 & 0 & -1 \\
  7 & 0 & 0 & -1 & 1 & 0 \\
\end{block}
\end{blockarray}
\quad .\]

\noindent The Cartan counterpart is of finite type $D_5$. This can be verified by 
applying the permutation $\sigma = (2, 1, 5, 4, 7)$, which transforms it into the standard generalized Cartan matrix of type $D_5$.\\

\item We consider now the choice of $J=\{1,2,3,4\}$ for $A_4$. Here, we can choose the bold part to be the whole standard longest word $w_0$:
$$ w_0=\boldsymbol{s_1s_2s_3s_4s_1s_2s_3s_1s_2s_1 }$$
It is straightforward to verify:
$$s(k)= \infty \iff k \in \{4,7,9,10 \}.$$
It follows that the mutable part of the exchange matrix is

\[
\begin{blockarray}{ccccccc}
 & 1 & 2 & 3 & 5 & 6 & 8\\
\begin{block}{c(cccccc)}
  1& 0 & -1 & 0 & 1 & 0 & 0 \\
  2 & 1 & 0 & -1 & -1 & 1 & 0 \\
  3 & 0 & 1 & 0 & 0 & -1 & 0 \\
  5 & -1 & 1 & 0 & 0 & -1 & 1 \\
  6 & 0 & -1 & 1 & 1& 0 & -1 \\
  8 & 0 & 0 & 0 & -1 & 1 & 0 \\
\end{block}
\end{blockarray}
\quad .\]

\noindent Now, we can see that the underlying graph consists of a central triangle $\{2, 5, 6\}$ where every edge is shared with an adjacent peripheral triangle (namely $\{1, 2, 5\}$, $\{2, 3, 6\}$, and $\{5, 6, 8\}$). A single mutation at vertex $2$ transforms this quiver into a tree graph isomorphic to the standard Dynkin diagram of type $D_6$. Thus, the cluster algebra is of finite type $D_6$.\\

We move now to the case where $n=5$, namely, $A_5$:\\

\item Now, consider $A_5$ and $J=\{3\}$. We may write $w_0$ as
$$w_0=s_1s_2s_1s_4s_5s_4\boldsymbol{s_3s_4s_2s_3s_5s_4s_1s_2s_3}.$$
The frozen variables of the bold part are labeled by:
$$s(k)=\infty \iff k \geq 5 .$$
Thus, the mutable part of the exchange matrix is

\[
\begin{blockarray}{ccccc}
 & 1 & 2 & 3 & 4\\
\begin{block}{c(cccc)}
  1& 0 & -1 & -1 & 1  \\
  2 & 1 & 0 & 0 & -1  \\
  3 & 1 & 0 & 0 & -1  \\
  4 & -1 & 1 & 1 & 0  \\
\end{block}
\end{blockarray}
\quad .\]

\noindent Here, it is easy to verify that applying a matrix mutation at $k=1$ gives a matrix whose Cartan Counterpart is $D_4$.\\

\item For $A_5$ and $J=\{1,3\}$. We may write $w_0$ as
$$w_0=s_2s_4s_5s_4\boldsymbol{s_1s_2s_3s_4s_2s_3s_5s_4s_1s_2s_3}.$$
The frozen variables of the bold part are labeled by:
$$s(k)=\infty \iff k \geq 7 .$$
Thus, the mutable part of the exchange matirx is

\[
\begin{blockarray}{ccccccc}
 & 1 & 2 & 3 & 4 & 5 & 6\\
\begin{block}{c(cccccc)}
  1& 0 & 0 & 0 & 0 & -1 & 0 \\
  2 & 0 & 0 & -1 & 0 & 1 & 0 \\
  3 & 0 & 1 & 0 & -1 & -1 & 1 \\
  4 & 0 & 0 & 1 & 0 & 0 & -1 \\
  5 & 1 & -1 & 1 & 0 & 0 & -1 \\
  6 & 0 & 0 & -1 & 1 & 1 & 0 \\
\end{block}
\end{blockarray}
\quad .\]

\noindent It is an easy exercise now to verify that the mutation at $k=3$ followed by the mutation at $k=5$ gives a matrix whose Cartan type is $E_6$. Therefore, the induced cluster algebra is of type $E_6$.\\

\item For $A_5$ and $J=\{2,3\}$. The longest word $w_0$ can be expressed as
$$w_0=s_2s_4s_5s_4\boldsymbol{s_2s_1s_3s_4s_2s_3s_5s_4s_1s_2s_3}.$$
The frozen variables of the bold part are labeled by:
$$s(k)=\infty \iff k \geq 7 .$$
Thus, the mutable part of the exchange matirx is

\[
\begin{blockarray}{ccccccc}
 & 1 & 2 & 3 & 4 & 5 & 6\\
\begin{block}{c(cccccc)}
  1& 0 & -1 & -1 & 0 & 1 & 0 \\
  2 & 1 & 0 & 0 & 0 & -1 & 0 \\
  3 & 1 & 0 & 0 & -1 & -1 & 1 \\
  4 & 0 & 0 & 1 & 0 & 0 & -1 \\
  5 & -1 & 1 & 1 & 0 & 0 & -1 \\
  6 & 0 & 0 & -1 & 1 & 1 & 0 \\
\end{block}
\end{blockarray}
\quad .\]

\noindent Here again, one can easily verify that the mutation at $k=1$ followed by the mutation at $k=3$ gives a matrix whose Cartan type is $E_6$. Hence, the induced cluster algebra is of type $E_6$.\\

\item Consider now $A_5$ with $J=\{1,2,3\}$. Express $w_0$ as
$$w_0=s_4s_5s_4\boldsymbol{s_2s_1s_2s_3s_4s_2s_3s_5s_4s_1s_2s_3}.$$
The frozen variables of the bold part are labeled by:
$$s(k)=\infty \iff k \geq 8 .$$
Following this, we find that the mutable part of the initial exchange matrix is

\[
\begin{blockarray}{cccccccc}
& 1 & 2 & 3 & 4 & 5 & 6 & 7\\
\begin{block}{c(ccccccc)}
1& 0 & -1 & 1 & 0 & 0 & 0 & 0 \\
2& 1 & 0 & 0 & 0 & 0 & -1 & 0 \\
3 & -1 & 0 & 0 & -1 & 0 & 1 & 0 \\
4 & 0 & 0 & 1 & 0 & -1 & -1 & 1 \\
5 & 0 & 0 & 0 & 1 & 0 & 0 & -1 \\
6 & 0 & 1 & -1 & 1 & 0 & 0 & -1 \\
7 & 0 & 0 & 0 & -1 & 1 & 1 & 0 \\
\end{block}
\end{blockarray}
\quad .\]

\noindent We may use here the mutatuion sequence $\mu_4 \circ \mu_3 \circ \mu_6$ and see that the type of this cluster algebra is $E_7$.\\

We consider now $A_6$:

\item If the group is of type $A_6$ and $J=\{3\}$, then the longest word can be expressed as
$$w_0=s_1s_2s_1s_4s_5s_6s_4s_5s_4\boldsymbol{s_3s_4s_5s_6s_2s_3s_4s_5s_1s_2s_3s_4}.$$

\noindent Considering the bold part, as usual, we get the following indices for frozen variables:
$$s(k)=\infty \iff \bigg( k=4 \textnormal{ or } k \geq 8 \bigg).$$
Thus, the mutable part of the initial exchange matrix is

\[
\begin{blockarray}{ccccccc}
 & 1 & 2 & 3 & 5 & 6 & 7\\
\begin{block}{c(cccccc)}
  1& 0 & -1 & 0 & -1 & -1 & 0 \\
  2 & 1 & 0 & -1 & 0 & -1 & -1 \\
  3 & 0 & 1 & 0 & 0 & 0 & -1 \\
  5 & 1 & 0 & 0 & 0 & -1 & 0 \\
  6 & 1 & 1 & 0 & 1 & 0 & -1 \\
  7 & 0 & 1 & 1 & 0 & 1 & 0 \\
\end{block}
\end{blockarray}
\quad .\]

\noindent Now, if we apply the mutations $\mu_1 \circ \mu_4$, then the resulting matrix has a Cartan counterpart of type $E_6$.\\

\item In the case of $A_6$ and $J=\{2,3\}$, the longest word can be expressed as
$$w_0=s_1s_4s_5s_6s_4s_5s_4\boldsymbol{s_2s_1s_3s_4s_5s_6s_2s_3s_4s_5s_1s_2s_3s_4}.$$

\noindent From the bold part, we get that the frozen variables are indexed by
$$s(k)=\infty \iff \bigg( k=6 \textnormal{ or } k \geq 10 \bigg).$$
Thus, the mutable part of the initial exchange matrix is

\[
\begin{blockarray}{ccccccccc}
 & 1 & 2 & 3 & 4 & 5 & 7 & 8 & 9\\
\begin{block}{c(cccccccc)}
1& 0 & -1 & -1 & 0 & 0 & -1 & 0 & 0 \\
2& 1 & 0 & 0 & 0 & 0 & -1 & 0 & 0 \\
3& 1 & 0 & 0 & -1 & 0 & -1 & -1 & 0 \\
4 & 0 & 0 & 1 & 0 & -1 & 0 & -1 & -1 \\
5 & 0 & 0 & 0 & 1 & 0 & 0 & 0 & -1 \\
7 & 1 & 1 & 1 & 0 & 0 & 0 & -1 & 0 \\
8 & 0 & 0 & 1 & 1 & 0 & 1 & 0 & -1 \\
9 & 0 & 0 & 0 & 1 & 1 & 0 & 1 & 0 \\
\end{block}
\end{blockarray}
\quad .\]

\noindent Now, if we apply the sequence of mutations at $k=6$ followed by $k=7$, then we get a matrix whose Cartan counterpart -after a permutation of rows and columns- is of type $E_8$.\\

Finally, we do the following case for type $A_7$:\\

\item Consider $A_7$ and $J=\{3\}.$ The longest word here can be represented as
$$w_0=s_1s_2s_1s_4s_5s_6s_7s_4s_5s_6s_4s_5s_4\boldsymbol{s_3s_4s_5s_6s_7s_2s_3s_4s_5s_6s_1s_2s_3s_4s_5}.$$

\noindent The frozen variables are indexed by

$$s(k)=\infty \iff \bigg( k=5 \textnormal{ or } k \geq 10 \bigg).$$
From this, it follows that the mutable part of the exchange matrix is

\[
\begin{blockarray}{ccccccccc}
 & 1 & 2 & 3 & 4 & 6 & 7 & 8 & 9\\
\begin{block}{c(cccccccc)}
1& 0 & -1 & 0 & 0 & -1 & -1 & 0 & 0 \\
2& 1 & 0 & -1 & 0 & 0 & -1 & -1 & 0 \\
3& 0 & 1 & 0 & -1 & 0 & 0 & -1 & -1 \\
4 & 0 & 0 & 1 & 0 & 0 & 0 & 0 & -1 \\
6 & 1 & 0 & 0 & 0 & 0 & -1 & 0 & 0 \\
7 & 1 & 1 & 0 & 0 & 1 & 0 & -1 & 0 \\
8 & 0 & 1 & 1 & 0 & 0 & 1 & 0 & -1 \\
9 & 0 & 0 & 1 & 1 & 0 & 0 & 1 & 0 \\
\end{block}
\end{blockarray}
\quad .\]

\noindent Mutating the matrix at $k=6$ followed by $k=7$ yields a matrix whose Cartan counterpart is of type $E_8$. Hence, the cluster algebra is of finite Cartan type $E_8$.
\end{enumerate}

This ends the finite type classification for type $A_n$.

\subsection{Type $B_n$ and Type $C_n$}
As before, we split this subsection into several cases. Throughout, we use the convention that vertex $n$ corresponds to the long root in type $B_n$ and to the short root in type $C_n$. This is the opposite of the convention used by GLS, and therefore our labeling is obtained from theirs by swapping vertices $1$ and $n$.

\begin{enumerate}
    \item If the group $G$ is of type $B_3$ (or $C_3$) and $J=\{3\}$.
    The longest word can be expressed as follows
\[w_0={s_1s_2s_1} \boldsymbol{s_3s_2s_1s_3s_2s_3}.\]

\noindent We have $s(3)=s(5)=s(6)=\infty$ and $s(k) \neq \infty$ for $k\in \{1,2,4 \}$. Thus, the mutable variables are indexed by $1,2,4$, and the frozen variables are indexed by $3,5,6$. One can verify that the mutable part of the exchange matrix is
\[
\begin{blockarray}{cccc}
& 1 & 2 & 4 \\
\begin{block}{c(ccc)}
  1 & 0 & -2 & 1 \\
  2 & 1 & 0 & -1 \\
  4 & -1 & 2 & 0 \\
\end{block}
\end{blockarray}
 \quad . \] 

\noindent Mutate the initial exchange matrix at $k=2$ to obtain an exchange matrix whose type is $B_3 \textnormal{ or }C_3$. Therefore, the cluster algebra is of finite Cartan type $B_3$ (or equivalently $C_3$, as their cluster complexes are isomorphic).\\

\item For the case of $B_2=C_2$ and $J=\{1,2\}$. The longest word can be expressed as
$$w_0=\boldsymbol{s_1s_2s_1s_2}.$$
It is easily seen that the frozen variables are the ones indexed by $k=3$ and $k=4$. Also, the exchange matrix is
\[
\begin{blockarray}{ccc}
 & 1 & 2 \\
\begin{block}{c(cc)}
  1& 0 & -1  \\
  2 & 2 & 0  \\
\end{block}
\end{blockarray}
\quad .\]

\noindent This is a matrix whose Counter Cartan type is $B_2=C_2$. Hence, this also gives a finite cluster algebra whose type is $B_2=C_2$.
\\

\item For $B_n$ (or $C_n$) and $J=\{1\}$, we may express the longest word as
$$w_0=u \cdot \boldsymbol{s_1s_2s_3...s_ns_{n-1}s_{n-2}...s_1},$$
where $u$ is the standard longest word of type $B_{n-1}$ indexed by $2,...,n$. The frozen variables are indexed by
$$s(k)= \infty \iff k \geq n.$$
Clearly, the principal mutable part is the zero matrix of size $n-1$. Hence, this cluster algebra is of type $(A_1)^{n-1}$.
This finishes the finite type classification for type $B_n$ and $C_n$.
\end{enumerate}

\subsection{Type $D_n$} Like the previous subsections, we divide this also to three cases. We label the Dynkin diagram of type $D_n$ so that the two vertices of the fork are $n-1$ and $n$. This is different from the convention used by GLS, whose fork vertices are $1$ and $2$.

\begin{enumerate}
    \item For the case of $D_n$ and $J=\{1\}$, we may express $w_0$ as follows
    $$w_0=s_ns_{n-1}(s_{n-2}s_ns_{n-1}s_{n-2})(s_{n-3}s_{n-2}s_ns_{n-1}s_{n-2}s_{n-3})...\boldsymbol{(s_1s_2...s_{n-2}s_ns_{n-1}s_{n-2}...s_1)}.$$
    Clearly, the frozen variables are indexed by
    $$s(k)=\infty \iff k \geq n-1.$$
    Also, the mutable part of this cluster algebra is the zero matrix of size $n-2$. Hence, this is a finite type cluster algebra of size $(A_1)^{n-2}$.\\
    
    \item For $D_4$ and $J=\{3,4\}$, the longest word can be expressed as
    $$w_0=s_1s_2s_1 \boldsymbol{s_4s_3s_2s_1s_4s_3s_2s_4s_3}.$$
    Clearly,
    $$s(k)=\infty \iff k \in \{4,7,8,9 \}.$$
    It follows that the principal mutable part of the initial exchange matrix is

\[
\begin{blockarray}{cccccc}
 & 1 & 2 & 3 & 5 & 6 \\
\begin{block}{c(ccccc)}
  1& 0 & 0 & -1 & 1 & 0  \\
  2 & 0 & 0 & -1 & 0 & 1 \\
  3 & 1 & 1 & 0 & -1 & -1 \\
  5 & -1 & 0 & 1 & 0 & 0\\
  6 & 0 & -1 & 1 & 0 & 0 \\
\end{block}
\end{blockarray}
\quad .\]
Here, one can see that the mutation at $k=3$ gives a mutation-equivalent matrix to an orientation of the Cartan matrix of type $A_5$. Hence, this is cluster algebra whose type is $A_5$.\\
    
    \item For $D_5$ and $J=\{5\}$, we can use
    $$w_0=s_1s_2s_3s_4s_1s_2s_3s_1s_2s_1\boldsymbol{s_5s_3s_2s_1s_4s_3s_2s_5s_3s_4}.$$
Now, we can see that the mutable variables are indexed by
$$s(k)=\infty \iff k \in \{4,7,8,9,10 \}.$$
Thus, the principal mutable part of the exchange matrix is

\[
\begin{blockarray}{cccccc}
 & 1 & 2 & 3 & 5 & 6 \\
\begin{block}{c(ccccc)}
  1& 0 & 0 & 0 & 0 & -1  \\
  2 & 0 & 0 & -1 & -1 & 1 \\
  3 & 0 & 1 & 0 & 0 & -1 \\
  5 & 0 & 1  & 0 & 0 & 0 \\
  6 & 1 & -1 & 1 & 0 & 0 \\
\end{block}
\end{blockarray}
\quad .\]
It is not hard now to see that the mutation at $k=5$ followed by an index permutation will give the $A_5$ matrix. Hence, the cluster algebra is of type $A_5$.

\end{enumerate}

\section{Consequences}
Following the previous section, one can immediately get that although the constructions of GLS and GY are based on different settings (the first is representation-theoretic and the second is Poisson-geometric), we can get significant results. We start this section by the following remark:

\begin{remark}
In Section 9 of \cite{GLS1}, one can see in 9.3.2 that the cluster algebra constructed by GLS for the coordinate ring of the Schubert cell of the simply laced types matches the construction of GY that we mentioned in Theorem \ref{seed construction}. Therefore, the following result follows.
\end{remark}

\begin{theorem}
The cluster algebra constructions by GLS and by GY on the coordinate ring of Schubert cells are equal.
\end{theorem}

\begin{remark}
For the arbitrary types of the group $G$, one can see that the exchange matrix of \ref{seed construction} above and the one of Proposition 4.2 of \cite{D} are equal.
\end{remark}

As a direct consequence of the previous remark and the previous section, we get the following two corollaries.

\begin{corollary}
The cluster algebra constructions of \cite{D} and \cite{Kadhem} on the coordinate ring of partial flag varieties have the same finite type classifications.
\end{corollary}

\begin{corollary}
The localizations of cluster algebra constructions of \cite{D} and \cite{Kadhem} on the coordinate ring of partial flag varieties coincide after localization by the minors indexed by $J$. 
\end{corollary}

\end{document}